\documentclass[reqno, 11pt, a4paper]{amsart}
\usepackage{kotex}
\usepackage{amsmath}
\usepackage{a4wide}
\usepackage{enumerate}
\usepackage[
en-US,
showzone=false,
]{datetime2}

\usepackage{color}
\usepackage{ifpdf}
  \usepackage{graphicx}
\ifpdf 
  \DeclareGraphicsExtensions{.pdf,.png,.mps}
  \usepackage{pgf}
  \usepackage{epic}
\else 
  \usepackage{graphicx}
  \DeclareGraphicsExtensions{.eps,.bmp}
  \usepackage{pgf}
\fi
\usepackage{bez123}
\usepackage{wrapfig}

\newtheorem{thm}{Theorem}
\newtheorem{cor}[thm]{Corollary}
\newtheorem{lem}[thm]{Lemma}

\theoremstyle{definition}

\newcommand{\nt}[1]{\marginpar{\tiny #1}}
\renewcommand{\nt}[1]{}
\newcommand{\set}[1]{\left\{#1\right\}}

\newcommand{\T}{\mathcal{T}}
\newcommand{\Tn}{\T_n}
\newcommand{\Tnd}{\Tn^{(d)}}

\newcommand{\FC}{\mathcal{FC}}
\newcommand{\FCn}{\FC_n}
\newcommand{\FCnd}{\FCn^{(d)}}
\newcommand{\oFCnd}{\overline{\FC}{}_n^{(d)}}

\newcommand{\V}{\mathcal{V}}
\newcommand{\Vn}{\V_n}
\newcommand{\Vnd}{\Vn^{(d)}}

\DeclareMathOperator{\Cat}{Cat}

\begin{document}

\author{Sangwook Kim}
\address[Sangwook Kim]{Department of Mathematics, Chonnam National University, Gwangju 61186, Korea}
\email{swkim.math@jnu.ac.kr}

\author{Seunghyun Seo}
\address[Seunghyun Seo]{Department of Mathematics Education, Kangwon National University, Chuncheon 24341, Korea}
\email{shyunseo@kangwon.ac.kr}

\author{Heesung Shin$^\dag$}
\address[Heesung Shin]{Department of Mathematics, Inha University, Incheon 22212, Korea}
\email{shin@inha.ac.kr}
\thanks{\dag Corresponding author}

\title[Refined enumeration of vertices among all rooted ordered $d$-trees]%
{Refined enumeration of vertices among all rooted ordered $d$-trees}

\date{\DTMnow}

\begin{abstract}
In this paper we enumerate the cardinalities for the set of all vertices of outdegree $\ge k$ at level $\ge \ell$ among all rooted ordered $d$-trees with $n$ edges.
Our results unite and generalize several previous works in the literature.
\end{abstract}

\maketitle

\section{Introduction}
\label{sec:intro}

For a positive integer $d$, the $n$th $d$-Fuss-Catalan number is given by
$$
\Cat_n^{(d)} = \frac{1}{dn + 1} \binom{(d+1)n}{n},\quad\text{for $n\ge0$.}
$$
They are generalization of well-known Catalan numbers.
Like Catalan numbers, there are several classes which are enumerated by Fuss-Catalan numbers.
Most well-known class is Fuss-Catalan paths.
A $d$-Fuss-Catalan path of length $(d+1)n$ is a lattice path from $(0,0)$ to $((d+1)n, 0)$
using up steps $(1,d)$ and down steps $(1,-1)$ such that it stays weakly above the $x$-axis.
Denote by $\FCnd$ the set of $d$-Fuss-Catalan paths of length $(d+1)n$.
Another example is dissections of a $(dn+2)$-gon into $(d+2)$-gons by diagonals.
There are
three more classes which are enumerated by $d$-Fuss-Catalan numbers.

\subsection*{Rooted ordered $d$-trees}
A rooted tree can be considered as a process of successively gluing an edge (1-simplex) to a vertex (0-simplex) from the root in a half-plane,
where the root is fixed in the line (1-dimensional hyperplane)
as the boundary of the given half-plane.
In same way, we can define a \emph{rooted $d$-tree} in $(d+1)$-dimensional lower Euclidean half-space $\mathbb{R}^{d+1}_{-}$ as follows:
The root $\mathbf{r}$ is a $(d-1)$-simplex fixed in the boundary of $\mathbb{R}^{d+1}_{-}$.
From the root $(d-1)$-simplex $\mathbf{r}$, we glue $d$-simplices (as edges) successively to one of previous $(d-1)$-simplices (as vertices) in $\mathbb{R}^{d+1}_{-}$.
(See \cite[$d$-dimensional trees]{BP69}.)
By definition, if $d=1$, a rooted $d$-tree is a rooted tree.

In a rooted tree, we can consider a \emph{linear order} among all edges having one common vertex by their positions and such a tree is called a \emph{rooted ordered tree}. Similarly, in higher dimensional cases, we can also give a linear order among $d$-simplices having one common $(d-1)$-simplices \emph{naturally} by their positions and such a tree is also called a \emph{rooted ordered $d$-tree}.
Jani, Rieper and Zeleke \cite{JRZ02} enumerate ordered $K$-trees,
which is obained in a similar way using $d$-simplices with $d \in K$.
\nt{They are using $(d+1)$-gons instead of $d$-simplices.
Also, they called $(d+1)$-trees instead of $d$-trees.}

\subsection*{Rooted $d$-ary cacti}
A \emph{cactus} is a connected simple graph in which each edge is
contained in exactly one elementary cycle.
These graphs are also known as ``Husimi trees''.
They are introduced by Harary and Uhlenbeck \cite{HU53}.
If each elementary cycle has exactly $d$ edges, a cactus is called
a $d$-ary cactus.
B\'{o}na et al. \cite{BBLL00} provide enumerations of various classes of $d$-ary cacti.

\subsection*{Rooted $d$-tuplet trees}
Instead of $d$-simplices used in rooted ordered $d$-trees, we may use $(d+1)$-gons. 
A root is a vertex fixed in the bounding hyperplane of a half-plane.
One can glue $(d+1)$-gons to a vertex from the root.
A tree obtained in this way is called a \emph{rooted $d$-tuplet tree} and the $(d+1)$-gons are called \emph{$d$-tuplets}.
As there is a \emph{linear order} on the vertices in a tuplet,
one can show that there is a one-to-one correspondence
between rooted ordered $d$-trees with $n$ edges 
and rooted $d$-tuplet trees with $n$ tuplets.
Thus rooted ordered $d$-trees and rooted $d$-tuplet trees are essentially the same.
Note that the underlying graph of a $d$-tuplet tree is a $(d+1)$-ary cactus.

Let $\mathcal{T}_n^{(d)}$ be the set of rooted $d$-tuplet trees with $n$ tuplets.
It is easy to see that the cardinality of $\mathcal{T}_n^{(d)}$ is the $n$th $d$-Fuss-Catalan number $\Cat_{n}^{(d)}$.
For example, there are 22 rooted $3$-tuplet trees with $3$ tuplets, see Figure~\ref{fig:tree}.
\begin{figure}[t]
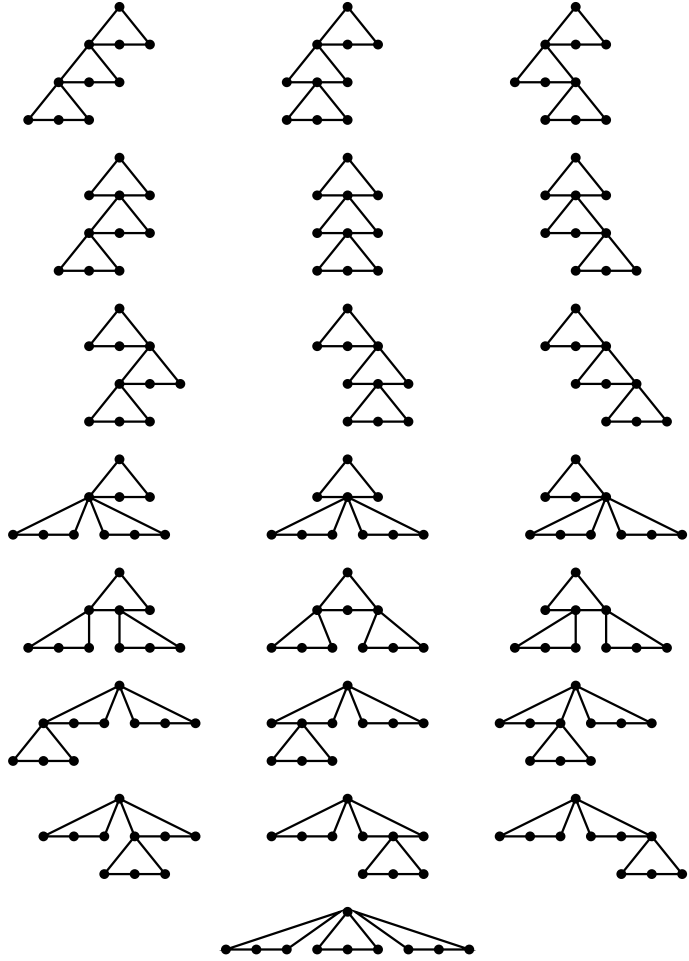

\centering
%
\caption{All rooted $3$-tuplet trees with 3 tuplets in $\mathcal{T}_3^{(3)}$}
\label{fig:tree}
\end{figure}
Clearly the number of vertices among rooted $d$-tuplet tree with $n$ tuplets in $\Tnd$ is
\begin{align}
(dn+1)\Cat_{n}^{(d)} = \binom{(d+1)n}{n}.
\label{eq:all}
\end{align}


In a rooted $d$-tuplet tree,
the \emph{degree} of a vertex is the number of tuplets it connects.
We can have the notion of the \emph{outdegree} of a vertex $v$, which is the number of tuplets starting at $v$ and pointing away from the root.
The \emph{level} of a vertex $v$ in a rooted $d$-tuplet tree is the
distance (number of tuplets) from the root to $v$.
Table~\ref{table:vertices} shows
the number of all vertices of outdegree $k$ at level $\ell$ among all rooted $3$-tuplet trees in $\T_3^{(3)}$.
\begin{table}[t]
\begin{tabular}{c|cccc|c}
$\ell \backslash k$&0&1&2&3 & $\sum$\\ \hline
0&  0&  15& 6&  1&  22\\
1&  66& 21& 3&  0&  90\\
2&  72& 9&  0&  0&  81\\
3&  27& 0&  0&  0&  27\\ \hline
$\sum$ & 165 & 45 & 9 & 1 & 220
\end{tabular}
\caption{The number of vertices of outdegree $k$ at level $\ell$ among all rooted $3$-tuplet trees in $\T_3^{(3)}$}
\label{table:vertices}
\end{table}
For example, there are 9 vertices of outdegree 1 at level 2 in $\T_3^{(3)}$,
see Figure~\ref{fig:tree}.

In a rooted $d$-tuplet tree,
there exists the unique vertex $u$ in each tuplet
such that its level is less than levels of the other vertices $v_1, \dots, v_{d}$.
Here, $u$ is called the \emph{parent} of $v_i$'s and each $v_i$ is called a \emph{child} of $u$.
For each vertex $v$ (except the root), there exists the unique tuplet containing $v$ toward the root, called the \emph{tuplet of $v$}.
Vertices with the same parent are called \emph{siblings}.
For two silbings $v$ and $w$, if $v$ is on the left of $w$,
$v$ is called an \emph{elder} sibling of $w$ and $w$ is called a \emph{younger} sibling of $v$.

Recently Eu, Seo, and Shin \cite{ESS17} gave a formula for the number of vertices among all trees in the set of rooted ordered trees under some conditions.
\begin{thm}[Eu, Seo, and Shin, 2017]
\label{thm:ESS17}
Given $n \ge 1$,
for any nonnegative integers $k$ and $\ell$,
the number of all vertices of outdegree $\ge k$ at level $\ge \ell$
among all rooted ordered trees with $n$ edges is
\begin{align}
\label{eq:lem}
\binom{2n -k }{ n+\ell}.
\end{align}
\end{thm}

We give a generalization of the formula \eqref{eq:lem} for $\Tnd$ by generalizing their bijection.
\begin{thm}[Main Result]
\label{thm:main1}
Given $n \ge 1$, for any nonnegative integers $k$ and $\ell$, the number of all vertices of outdegree $\ge k$ at level $\ge \ell$
among all rooted $d$-tuplet trees with $n$ tuplets is
\begin{align}
\label{eq:KSS16}
d^\ell \binom{(d+1)n-k}{dn+\ell}.
\end{align}
\end{thm}

We also find a refinement of the formula \eqref{eq:KSS16}.
\begin{thm}
\label{thm:main2}
Given $n \ge 1$, for any two nonnegative integers $i$, $j$, one nonnegative integer $k$ which is a multiple of $d$, and one positive integer $\ell$,
the number of all vertices
among all rooted $d$-tuplet trees with $n$ tuplets
such that
\begin{itemize}
\item having at least $i$ elder siblings,
\item having at least $j$ younger siblings,
\item having at least $k$ children,
\item at level $\ge \ell$
\end{itemize}
is
\begin{align}
d^\ell
\left(1- \frac{\beta}{d} \frac{dn+\ell}{(d+1)n-\alpha}\right)
\binom{(d+1)n-\alpha}{dn+\ell},
\label{eq:refinement}
\end{align}
where $\alpha$ and $\beta$ are nonnegative integers satisfying $i+j+k = \alpha d + \beta$ and $0\le \beta < d$.
\end{thm}



The rest of the paper is organized as follows.
In Section~\ref{sec:mainproof},
we show the Theorem~\ref{thm:main1} bijectively.
In Section~\ref{sec:calculation},
we give a combinatorial proof of the Theorem~\ref{thm:main2}.
In Section~\ref{sec:coro},
we present corollaries induced from Theorem~\ref{thm:main1} and \ref{thm:main2}.

\nt{Shin: We must write the outline.}

\section{A bijective proof of Theorem~\ref{thm:main1}}
\label{sec:mainproof}
Henceforth, a \emph{tree} is assumed to be a rooted $d$-tuplet tree.
Let $\mathcal{V}$ be the set of pairs $(T,v)$ such that $v$ is a vertex of outdegree $\ge k$ at level $\ge \ell$ in $T \in \T_n^{(d)}$.
Let $\mathcal{P}$ be the set of sequences in $\set{0,\dots, d-1}$ of length $\ell$.
Let $\mathcal{L}$ be the set of \emph{lattice paths} of length $((d+1)n-k)$ from $(k, dk)$ to $((d+1)n, -(d+1)\ell)$, consisting of $(n-k-\ell)$ up-steps along the vector $(1,d)$ and $(dn+\ell)$ down-steps along the vector $(1,-1)$.
In order to show Theorem~\ref{thm:main1}, it is enough to construct a bijection $\Phi$ between $\mathcal{V}$ and $\mathcal{P} \times \mathcal{L}$, due to
\begin{align*}
\# \mathcal{P} &= d^{\ell}, &
\# \mathcal{L} &= \binom{(d+1)n - k}{n -k -\ell,~ dn+\ell} = \binom{(d+1)n -k }{ dn+\ell }.
\end{align*}

\subsection*{Three bijections $\varphi$, $\overline{\varphi}$, and $\psi$}
Let a \emph{reverse $d$-Fuss-Catalan path} of length $(d+1)n$ be a lattice path from $(0,0)$ to $((d+1)n, 0)$
using up steps $(1,d)$ and down steps $(1,-1)$ such that it stays weakly below the $x$-axis.
Denote by $\oFCnd$ the set of reverse $d$-Fuss-Catalan paths of length $(d+1)n$.

Before constructing the bijection $\Phi$, we introduce three bijections 
\begin{align*}
\varphi&: \Tnd \to \FCnd, & \overline{\varphi}&: \Tnd \to \oFCnd, & \psi&: \Tnd \to \FCnd.
\end{align*}

The bijection $\varphi$ corresponds a tree to a lattice path weakly above the $x$-axis
by recording the steps when the tree is traversed in preorder:
whenever we go down a side of a tuplet, record an up-step along the vector $(1,d)$ and
whenever we go right or up a side of a tuplet, record a down-step along the vector $(1,-1)$.

Similarly, the bijection $\overline{\varphi}$ corresponds a tree to a lattice path weakly below the $x$-axis
by recording the steps when the tree is traversed in preorder:
whenever we go down or right a side, record an down-step along the vector $(1,-1)$ and
whenever we go up a side, record a up-step along the vector $(1,d)$.
An example of two bijections $\varphi$ and $\overline{\varphi}$ is shown in Figure~\ref{fig:varphi}.
\begin{figure}[t]
\centering
\begin{pgfpicture}{1.50mm}{3.00mm}{116.00mm}{37.00mm}
\pgfsetxvec{\pgfpoint{1.00mm}{0mm}}
\pgfsetyvec{\pgfpoint{0mm}{1.00mm}}
\color[rgb]{0,0,0}\pgfsetlinewidth{0.30mm}\pgfsetdash{}{0mm}
\pgfmoveto{\pgfxy(45.00,24.00)}\pgflineto{\pgfxy(55.00,26.00)}\pgfstroke
\pgfmoveto{\pgfxy(55.00,26.00)}\pgflineto{\pgfxy(52.12,26.14)}\pgflineto{\pgfxy(52.39,24.76)}\pgflineto{\pgfxy(55.00,26.00)}\pgfclosepath\pgffill
\pgfmoveto{\pgfxy(55.00,26.00)}\pgflineto{\pgfxy(52.12,26.14)}\pgflineto{\pgfxy(52.39,24.76)}\pgflineto{\pgfxy(55.00,26.00)}\pgfclosepath\pgfstroke
\pgfputat{\pgfxy(50.00,27.00)}{\pgfbox[bottom,left]{\fontsize{11.38}{13.66}\selectfont \makebox[0pt]{$\varphi$}}}
\pgfcircle[fill]{\pgfxy(20.00,30.00)}{0.50mm}
\pgfcircle[stroke]{\pgfxy(20.00,30.00)}{0.50mm}
\pgfcircle[fill]{\pgfxy(20.00,20.00)}{0.50mm}
\pgfcircle[stroke]{\pgfxy(20.00,20.00)}{0.50mm}
\pgfcircle[fill]{\pgfxy(16.00,20.00)}{0.50mm}
\pgfcircle[stroke]{\pgfxy(16.00,20.00)}{0.50mm}
\pgfcircle[fill]{\pgfxy(24.00,20.00)}{0.50mm}
\pgfcircle[stroke]{\pgfxy(24.00,20.00)}{0.50mm}
\pgfcircle[fill]{\pgfxy(12.00,20.00)}{0.50mm}
\pgfcircle[stroke]{\pgfxy(12.00,20.00)}{0.50mm}
\pgfcircle[fill]{\pgfxy(8.00,20.00)}{0.50mm}
\pgfcircle[stroke]{\pgfxy(8.00,20.00)}{0.50mm}
\pgfcircle[fill]{\pgfxy(4.00,20.00)}{0.50mm}
\pgfcircle[stroke]{\pgfxy(4.00,20.00)}{0.50mm}
\pgfcircle[fill]{\pgfxy(28.00,20.00)}{0.50mm}
\pgfcircle[stroke]{\pgfxy(28.00,20.00)}{0.50mm}
\pgfcircle[fill]{\pgfxy(32.00,20.00)}{0.50mm}
\pgfcircle[stroke]{\pgfxy(32.00,20.00)}{0.50mm}
\pgfcircle[fill]{\pgfxy(36.00,20.00)}{0.50mm}
\pgfcircle[stroke]{\pgfxy(36.00,20.00)}{0.50mm}
\pgfmoveto{\pgfxy(20.00,30.00)}\pgflineto{\pgfxy(4.00,20.00)}\pgflineto{\pgfxy(12.00,20.00)}\pgfclosepath\pgfstroke
\pgfmoveto{\pgfxy(20.00,30.00)}\pgflineto{\pgfxy(16.00,20.00)}\pgflineto{\pgfxy(24.00,20.00)}\pgfclosepath\pgfstroke
\pgfmoveto{\pgfxy(20.00,30.00)}\pgflineto{\pgfxy(28.00,20.00)}\pgflineto{\pgfxy(36.00,20.00)}\pgfclosepath\pgfstroke
\pgfcircle[fill]{\pgfxy(14.00,10.00)}{0.50mm}
\pgfcircle[stroke]{\pgfxy(14.00,10.00)}{0.50mm}
\pgfcircle[fill]{\pgfxy(10.00,10.00)}{0.50mm}
\pgfcircle[stroke]{\pgfxy(10.00,10.00)}{0.50mm}
\pgfcircle[fill]{\pgfxy(6.00,10.00)}{0.50mm}
\pgfcircle[stroke]{\pgfxy(6.00,10.00)}{0.50mm}
\pgfcircle[fill]{\pgfxy(18.00,10.00)}{0.50mm}
\pgfcircle[stroke]{\pgfxy(18.00,10.00)}{0.50mm}
\pgfcircle[fill]{\pgfxy(22.00,10.00)}{0.50mm}
\pgfcircle[stroke]{\pgfxy(22.00,10.00)}{0.50mm}
\pgfcircle[fill]{\pgfxy(26.00,10.00)}{0.50mm}
\pgfcircle[stroke]{\pgfxy(26.00,10.00)}{0.50mm}
\pgfmoveto{\pgfxy(16.00,20.00)}\pgflineto{\pgfxy(18.00,10.00)}\pgflineto{\pgfxy(26.00,10.00)}\pgfclosepath\pgfstroke
\pgfmoveto{\pgfxy(16.00,20.00)}\pgflineto{\pgfxy(6.00,10.00)}\pgflineto{\pgfxy(14.00,10.00)}\pgfclosepath\pgfstroke
\pgfcircle[fill]{\pgfxy(30.00,10.00)}{0.50mm}
\pgfcircle[stroke]{\pgfxy(30.00,10.00)}{0.50mm}
\pgfcircle[fill]{\pgfxy(34.00,10.00)}{0.50mm}
\pgfcircle[stroke]{\pgfxy(34.00,10.00)}{0.50mm}
\pgfcircle[fill]{\pgfxy(38.00,10.00)}{0.50mm}
\pgfcircle[stroke]{\pgfxy(38.00,10.00)}{0.50mm}
\pgfmoveto{\pgfxy(32.00,20.00)}\pgflineto{\pgfxy(30.00,10.00)}\pgflineto{\pgfxy(38.00,10.00)}\pgfclosepath\pgfstroke
\pgfmoveto{\pgfxy(60.00,20.00)}\pgflineto{\pgfxy(60.00,35.00)}\pgfstroke
\pgfmoveto{\pgfxy(60.00,35.00)}\pgflineto{\pgfxy(59.30,32.20)}\pgflineto{\pgfxy(60.00,35.00)}\pgflineto{\pgfxy(60.70,32.20)}\pgflineto{\pgfxy(60.00,35.00)}\pgfclosepath\pgffill
\pgfmoveto{\pgfxy(60.00,35.00)}\pgflineto{\pgfxy(59.30,32.20)}\pgflineto{\pgfxy(60.00,35.00)}\pgflineto{\pgfxy(60.70,32.20)}\pgflineto{\pgfxy(60.00,35.00)}\pgfclosepath\pgfstroke
\pgfmoveto{\pgfxy(58.00,22.00)}\pgflineto{\pgfxy(114.00,22.00)}\pgfstroke
\pgfmoveto{\pgfxy(114.00,22.00)}\pgflineto{\pgfxy(111.20,22.70)}\pgflineto{\pgfxy(114.00,22.00)}\pgflineto{\pgfxy(111.20,21.30)}\pgflineto{\pgfxy(114.00,22.00)}\pgfclosepath\pgffill
\pgfmoveto{\pgfxy(114.00,22.00)}\pgflineto{\pgfxy(111.20,22.70)}\pgflineto{\pgfxy(114.00,22.00)}\pgflineto{\pgfxy(111.20,21.30)}\pgflineto{\pgfxy(114.00,22.00)}\pgfclosepath\pgfstroke
\pgfcircle[fill]{\pgfxy(60.00,22.00)}{0.50mm}
\pgfcircle[stroke]{\pgfxy(60.00,22.00)}{0.50mm}
\pgfcircle[fill]{\pgfxy(62.00,28.00)}{0.50mm}
\pgfcircle[stroke]{\pgfxy(62.00,28.00)}{0.50mm}
\pgfcircle[fill]{\pgfxy(64.00,26.00)}{0.50mm}
\pgfcircle[stroke]{\pgfxy(64.00,26.00)}{0.50mm}
\pgfcircle[fill]{\pgfxy(66.00,24.00)}{0.50mm}
\pgfcircle[stroke]{\pgfxy(66.00,24.00)}{0.50mm}
\pgfcircle[fill]{\pgfxy(68.00,22.00)}{0.50mm}
\pgfcircle[stroke]{\pgfxy(68.00,22.00)}{0.50mm}
\pgfcircle[fill]{\pgfxy(70.00,28.00)}{0.50mm}
\pgfcircle[stroke]{\pgfxy(70.00,28.00)}{0.50mm}
\pgfcircle[fill]{\pgfxy(72.00,34.00)}{0.50mm}
\pgfcircle[stroke]{\pgfxy(72.00,34.00)}{0.50mm}
\pgfcircle[fill]{\pgfxy(74.00,32.00)}{0.50mm}
\pgfcircle[stroke]{\pgfxy(74.00,32.00)}{0.50mm}
\pgfcircle[fill]{\pgfxy(76.00,30.00)}{0.50mm}
\pgfcircle[stroke]{\pgfxy(76.00,30.00)}{0.50mm}
\pgfcircle[fill]{\pgfxy(78.00,28.00)}{0.50mm}
\pgfcircle[stroke]{\pgfxy(78.00,28.00)}{0.50mm}
\pgfcircle[fill]{\pgfxy(80.00,34.00)}{0.50mm}
\pgfcircle[stroke]{\pgfxy(80.00,34.00)}{0.50mm}
\pgfcircle[fill]{\pgfxy(82.00,32.00)}{0.50mm}
\pgfcircle[stroke]{\pgfxy(82.00,32.00)}{0.50mm}
\pgfcircle[fill]{\pgfxy(84.00,30.00)}{0.50mm}
\pgfcircle[stroke]{\pgfxy(84.00,30.00)}{0.50mm}
\pgfcircle[fill]{\pgfxy(86.00,28.00)}{0.50mm}
\pgfcircle[stroke]{\pgfxy(86.00,28.00)}{0.50mm}
\pgfcircle[fill]{\pgfxy(88.00,26.00)}{0.50mm}
\pgfcircle[stroke]{\pgfxy(88.00,26.00)}{0.50mm}
\pgfcircle[fill]{\pgfxy(90.00,24.00)}{0.50mm}
\pgfcircle[stroke]{\pgfxy(90.00,24.00)}{0.50mm}
\pgfcircle[fill]{\pgfxy(92.00,22.00)}{0.50mm}
\pgfcircle[stroke]{\pgfxy(92.00,22.00)}{0.50mm}
\pgfcircle[fill]{\pgfxy(94.00,28.00)}{0.50mm}
\pgfcircle[stroke]{\pgfxy(94.00,28.00)}{0.50mm}
\pgfcircle[fill]{\pgfxy(96.00,26.00)}{0.50mm}
\pgfcircle[stroke]{\pgfxy(96.00,26.00)}{0.50mm}
\pgfcircle[fill]{\pgfxy(98.00,32.00)}{0.50mm}
\pgfcircle[stroke]{\pgfxy(98.00,32.00)}{0.50mm}
\pgfcircle[fill]{\pgfxy(100.00,30.00)}{0.50mm}
\pgfcircle[stroke]{\pgfxy(100.00,30.00)}{0.50mm}
\pgfcircle[fill]{\pgfxy(102.00,28.00)}{0.50mm}
\pgfcircle[stroke]{\pgfxy(102.00,28.00)}{0.50mm}
\pgfcircle[fill]{\pgfxy(104.00,26.00)}{0.50mm}
\pgfcircle[stroke]{\pgfxy(104.00,26.00)}{0.50mm}
\pgfcircle[fill]{\pgfxy(106.00,24.00)}{0.50mm}
\pgfcircle[stroke]{\pgfxy(106.00,24.00)}{0.50mm}
\pgfcircle[fill]{\pgfxy(108.00,22.00)}{0.50mm}
\pgfcircle[stroke]{\pgfxy(108.00,22.00)}{0.50mm}
\pgfmoveto{\pgfxy(60.00,22.00)}\pgflineto{\pgfxy(62.00,28.00)}\pgflineto{\pgfxy(64.00,26.00)}\pgflineto{\pgfxy(66.00,24.00)}\pgflineto{\pgfxy(68.00,22.00)}\pgflineto{\pgfxy(70.00,28.00)}\pgflineto{\pgfxy(72.00,34.00)}\pgflineto{\pgfxy(74.00,32.00)}\pgflineto{\pgfxy(76.00,30.00)}\pgflineto{\pgfxy(78.00,28.00)}\pgflineto{\pgfxy(80.00,34.00)}\pgflineto{\pgfxy(82.00,32.00)}\pgflineto{\pgfxy(84.00,30.00)}\pgflineto{\pgfxy(86.00,28.00)}\pgflineto{\pgfxy(88.00,26.00)}\pgflineto{\pgfxy(90.00,24.00)}\pgflineto{\pgfxy(92.00,22.00)}\pgflineto{\pgfxy(94.00,28.00)}\pgflineto{\pgfxy(96.00,26.00)}\pgflineto{\pgfxy(98.00,32.00)}\pgflineto{\pgfxy(100.00,30.00)}\pgflineto{\pgfxy(102.00,28.00)}\pgflineto{\pgfxy(104.00,26.00)}\pgflineto{\pgfxy(106.00,24.00)}\pgflineto{\pgfxy(108.00,22.00)}\pgfstroke
\pgfmoveto{\pgfxy(60.00,18.00)}\pgflineto{\pgfxy(60.00,5.00)}\pgfstroke
\pgfmoveto{\pgfxy(60.00,5.00)}\pgflineto{\pgfxy(60.70,7.80)}\pgflineto{\pgfxy(60.00,5.00)}\pgflineto{\pgfxy(59.30,7.80)}\pgflineto{\pgfxy(60.00,5.00)}\pgfclosepath\pgffill
\pgfmoveto{\pgfxy(60.00,5.00)}\pgflineto{\pgfxy(60.70,7.80)}\pgflineto{\pgfxy(60.00,5.00)}\pgflineto{\pgfxy(59.30,7.80)}\pgflineto{\pgfxy(60.00,5.00)}\pgfclosepath\pgfstroke
\pgfmoveto{\pgfxy(58.00,16.00)}\pgflineto{\pgfxy(114.00,16.00)}\pgfstroke
\pgfmoveto{\pgfxy(114.00,16.00)}\pgflineto{\pgfxy(111.20,16.70)}\pgflineto{\pgfxy(114.00,16.00)}\pgflineto{\pgfxy(111.20,15.30)}\pgflineto{\pgfxy(114.00,16.00)}\pgfclosepath\pgffill
\pgfmoveto{\pgfxy(114.00,16.00)}\pgflineto{\pgfxy(111.20,16.70)}\pgflineto{\pgfxy(114.00,16.00)}\pgflineto{\pgfxy(111.20,15.30)}\pgflineto{\pgfxy(114.00,16.00)}\pgfclosepath\pgfstroke
\pgfcircle[fill]{\pgfxy(60.00,16.00)}{0.50mm}
\pgfcircle[stroke]{\pgfxy(60.00,16.00)}{0.50mm}
\pgfcircle[fill]{\pgfxy(62.00,14.00)}{0.50mm}
\pgfcircle[stroke]{\pgfxy(62.00,14.00)}{0.50mm}
\pgfcircle[fill]{\pgfxy(64.00,12.00)}{0.50mm}
\pgfcircle[stroke]{\pgfxy(64.00,12.00)}{0.50mm}
\pgfcircle[fill]{\pgfxy(66.00,10.00)}{0.50mm}
\pgfcircle[stroke]{\pgfxy(66.00,10.00)}{0.50mm}
\pgfcircle[fill]{\pgfxy(68.00,16.00)}{0.50mm}
\pgfcircle[stroke]{\pgfxy(68.00,16.00)}{0.50mm}
\pgfcircle[fill]{\pgfxy(70.00,14.00)}{0.50mm}
\pgfcircle[stroke]{\pgfxy(70.00,14.00)}{0.50mm}
\pgfcircle[fill]{\pgfxy(72.00,12.00)}{0.50mm}
\pgfcircle[stroke]{\pgfxy(72.00,12.00)}{0.50mm}
\pgfcircle[fill]{\pgfxy(74.00,10.00)}{0.50mm}
\pgfcircle[stroke]{\pgfxy(74.00,10.00)}{0.50mm}
\pgfcircle[fill]{\pgfxy(76.00,8.00)}{0.50mm}
\pgfcircle[stroke]{\pgfxy(76.00,8.00)}{0.50mm}
\pgfcircle[fill]{\pgfxy(78.00,14.00)}{0.50mm}
\pgfcircle[stroke]{\pgfxy(78.00,14.00)}{0.50mm}
\pgfcircle[fill]{\pgfxy(80.00,12.00)}{0.50mm}
\pgfcircle[stroke]{\pgfxy(80.00,12.00)}{0.50mm}
\pgfcircle[fill]{\pgfxy(82.00,10.00)}{0.50mm}
\pgfcircle[stroke]{\pgfxy(82.00,10.00)}{0.50mm}
\pgfcircle[fill]{\pgfxy(84.00,8.00)}{0.50mm}
\pgfcircle[stroke]{\pgfxy(84.00,8.00)}{0.50mm}
\pgfcircle[fill]{\pgfxy(86.00,14.00)}{0.50mm}
\pgfcircle[stroke]{\pgfxy(86.00,14.00)}{0.50mm}
\pgfcircle[fill]{\pgfxy(88.00,12.00)}{0.50mm}
\pgfcircle[stroke]{\pgfxy(88.00,12.00)}{0.50mm}
\pgfcircle[fill]{\pgfxy(90.00,10.00)}{0.50mm}
\pgfcircle[stroke]{\pgfxy(90.00,10.00)}{0.50mm}
\pgfcircle[fill]{\pgfxy(92.00,16.00)}{0.50mm}
\pgfcircle[stroke]{\pgfxy(92.00,16.00)}{0.50mm}
\pgfcircle[fill]{\pgfxy(94.00,14.00)}{0.50mm}
\pgfcircle[stroke]{\pgfxy(94.00,14.00)}{0.50mm}
\pgfcircle[fill]{\pgfxy(96.00,12.00)}{0.50mm}
\pgfcircle[stroke]{\pgfxy(96.00,12.00)}{0.50mm}
\pgfcircle[fill]{\pgfxy(98.00,10.00)}{0.50mm}
\pgfcircle[stroke]{\pgfxy(98.00,10.00)}{0.50mm}
\pgfcircle[fill]{\pgfxy(100.00,8.00)}{0.50mm}
\pgfcircle[stroke]{\pgfxy(100.00,8.00)}{0.50mm}
\pgfcircle[fill]{\pgfxy(102.00,6.00)}{0.50mm}
\pgfcircle[stroke]{\pgfxy(102.00,6.00)}{0.50mm}
\pgfcircle[fill]{\pgfxy(104.00,12.00)}{0.50mm}
\pgfcircle[stroke]{\pgfxy(104.00,12.00)}{0.50mm}
\pgfcircle[fill]{\pgfxy(106.00,10.00)}{0.50mm}
\pgfcircle[stroke]{\pgfxy(106.00,10.00)}{0.50mm}
\pgfcircle[fill]{\pgfxy(108.00,16.00)}{0.50mm}
\pgfcircle[stroke]{\pgfxy(108.00,16.00)}{0.50mm}
\pgfmoveto{\pgfxy(60.00,16.00)}\pgflineto{\pgfxy(62.00,14.00)}\pgflineto{\pgfxy(64.00,12.00)}\pgflineto{\pgfxy(66.00,10.00)}\pgflineto{\pgfxy(68.00,16.00)}\pgflineto{\pgfxy(70.00,14.00)}\pgflineto{\pgfxy(72.00,12.00)}\pgflineto{\pgfxy(74.00,10.00)}\pgflineto{\pgfxy(76.00,8.00)}\pgflineto{\pgfxy(78.00,14.00)}\pgflineto{\pgfxy(80.00,12.00)}\pgflineto{\pgfxy(82.00,10.00)}\pgflineto{\pgfxy(84.00,8.00)}\pgflineto{\pgfxy(86.00,14.00)}\pgflineto{\pgfxy(88.00,12.00)}\pgflineto{\pgfxy(90.00,10.00)}\pgflineto{\pgfxy(92.00,16.00)}\pgflineto{\pgfxy(94.00,14.00)}\pgflineto{\pgfxy(96.00,12.00)}\pgflineto{\pgfxy(98.00,10.00)}\pgflineto{\pgfxy(100.00,8.00)}\pgflineto{\pgfxy(102.00,6.00)}\pgflineto{\pgfxy(104.00,12.00)}\pgflineto{\pgfxy(106.00,10.00)}\pgflineto{\pgfxy(108.00,16.00)}\pgfstroke
\pgfmoveto{\pgfxy(45.00,14.00)}\pgflineto{\pgfxy(55.00,12.00)}\pgfstroke
\pgfmoveto{\pgfxy(55.00,12.00)}\pgflineto{\pgfxy(52.39,13.24)}\pgflineto{\pgfxy(52.12,11.86)}\pgflineto{\pgfxy(55.00,12.00)}\pgfclosepath\pgffill
\pgfmoveto{\pgfxy(55.00,12.00)}\pgflineto{\pgfxy(52.39,13.24)}\pgflineto{\pgfxy(52.12,11.86)}\pgflineto{\pgfxy(55.00,12.00)}\pgfclosepath\pgfstroke
\pgfputat{\pgfxy(50.00,15.00)}{\pgfbox[bottom,left]{\fontsize{11.38}{13.66}\selectfont \makebox[0pt]{$\overline{\varphi}$}}}
\end{pgfpicture}%
\caption{Two bijections $\varphi$ and $\overline{\varphi}$}
\label{fig:varphi}
\end{figure}

The bijection $\psi$ corresponds a tree to a lattice path weakly above the $x$-axis by recording the steps when the tree is traversed in preorder: whenever we meet a vertex of outdegree $m$, except the last leaf, record $m$ up-steps and $1$ down-step.
An example of the bijection $\psi$ is shown in Figure~\ref{fig:psi}.
\begin{figure}[t]
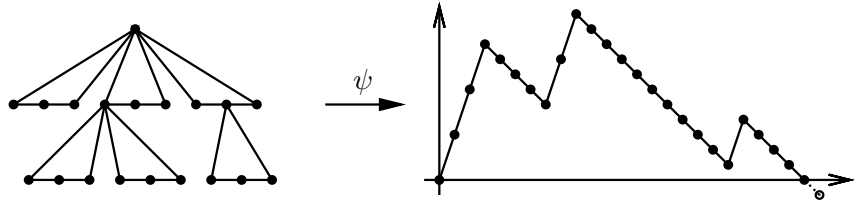

\centering
\begin{pgfpicture}{1.50mm}{5.50mm}{116.00mm}{35.00mm}
\pgfsetxvec{\pgfpoint{1.00mm}{0mm}}
\pgfsetyvec{\pgfpoint{0mm}{1.00mm}}
\color[rgb]{0,0,0}\pgfsetlinewidth{0.30mm}\pgfsetdash{}{0mm}
\pgfmoveto{\pgfxy(45.00,20.00)}\pgflineto{\pgfxy(55.00,20.00)}\pgfstroke
\pgfmoveto{\pgfxy(55.00,20.00)}\pgflineto{\pgfxy(52.20,20.70)}\pgflineto{\pgfxy(52.20,19.30)}\pgflineto{\pgfxy(55.00,20.00)}\pgfclosepath\pgffill
\pgfmoveto{\pgfxy(55.00,20.00)}\pgflineto{\pgfxy(52.20,20.70)}\pgflineto{\pgfxy(52.20,19.30)}\pgflineto{\pgfxy(55.00,20.00)}\pgfclosepath\pgfstroke
\pgfputat{\pgfxy(50.00,22.00)}{\pgfbox[bottom,left]{\fontsize{11.38}{13.66}\selectfont \makebox[0pt]{$\psi$}}}
\pgfcircle[fill]{\pgfxy(20.00,30.00)}{0.50mm}
\pgfcircle[stroke]{\pgfxy(20.00,30.00)}{0.50mm}
\pgfcircle[fill]{\pgfxy(20.00,20.00)}{0.50mm}
\pgfcircle[stroke]{\pgfxy(20.00,20.00)}{0.50mm}
\pgfcircle[fill]{\pgfxy(16.00,20.00)}{0.50mm}
\pgfcircle[stroke]{\pgfxy(16.00,20.00)}{0.50mm}
\pgfcircle[fill]{\pgfxy(24.00,20.00)}{0.50mm}
\pgfcircle[stroke]{\pgfxy(24.00,20.00)}{0.50mm}
\pgfcircle[fill]{\pgfxy(12.00,20.00)}{0.50mm}
\pgfcircle[stroke]{\pgfxy(12.00,20.00)}{0.50mm}
\pgfcircle[fill]{\pgfxy(8.00,20.00)}{0.50mm}
\pgfcircle[stroke]{\pgfxy(8.00,20.00)}{0.50mm}
\pgfcircle[fill]{\pgfxy(4.00,20.00)}{0.50mm}
\pgfcircle[stroke]{\pgfxy(4.00,20.00)}{0.50mm}
\pgfcircle[fill]{\pgfxy(28.00,20.00)}{0.50mm}
\pgfcircle[stroke]{\pgfxy(28.00,20.00)}{0.50mm}
\pgfcircle[fill]{\pgfxy(32.00,20.00)}{0.50mm}
\pgfcircle[stroke]{\pgfxy(32.00,20.00)}{0.50mm}
\pgfcircle[fill]{\pgfxy(36.00,20.00)}{0.50mm}
\pgfcircle[stroke]{\pgfxy(36.00,20.00)}{0.50mm}
\pgfmoveto{\pgfxy(20.00,30.00)}\pgflineto{\pgfxy(4.00,20.00)}\pgflineto{\pgfxy(12.00,20.00)}\pgfclosepath\pgfstroke
\pgfmoveto{\pgfxy(20.00,30.00)}\pgflineto{\pgfxy(16.00,20.00)}\pgflineto{\pgfxy(24.00,20.00)}\pgfclosepath\pgfstroke
\pgfmoveto{\pgfxy(20.00,30.00)}\pgflineto{\pgfxy(28.00,20.00)}\pgflineto{\pgfxy(36.00,20.00)}\pgfclosepath\pgfstroke
\pgfcircle[fill]{\pgfxy(14.00,10.00)}{0.50mm}
\pgfcircle[stroke]{\pgfxy(14.00,10.00)}{0.50mm}
\pgfcircle[fill]{\pgfxy(10.00,10.00)}{0.50mm}
\pgfcircle[stroke]{\pgfxy(10.00,10.00)}{0.50mm}
\pgfcircle[fill]{\pgfxy(6.00,10.00)}{0.50mm}
\pgfcircle[stroke]{\pgfxy(6.00,10.00)}{0.50mm}
\pgfcircle[fill]{\pgfxy(18.00,10.00)}{0.50mm}
\pgfcircle[stroke]{\pgfxy(18.00,10.00)}{0.50mm}
\pgfcircle[fill]{\pgfxy(22.00,10.00)}{0.50mm}
\pgfcircle[stroke]{\pgfxy(22.00,10.00)}{0.50mm}
\pgfcircle[fill]{\pgfxy(26.00,10.00)}{0.50mm}
\pgfcircle[stroke]{\pgfxy(26.00,10.00)}{0.50mm}
\pgfmoveto{\pgfxy(16.00,20.00)}\pgflineto{\pgfxy(18.00,10.00)}\pgflineto{\pgfxy(26.00,10.00)}\pgfclosepath\pgfstroke
\pgfmoveto{\pgfxy(16.00,20.00)}\pgflineto{\pgfxy(6.00,10.00)}\pgflineto{\pgfxy(14.00,10.00)}\pgfclosepath\pgfstroke
\pgfcircle[fill]{\pgfxy(30.00,10.00)}{0.50mm}
\pgfcircle[stroke]{\pgfxy(30.00,10.00)}{0.50mm}
\pgfcircle[fill]{\pgfxy(34.00,10.00)}{0.50mm}
\pgfcircle[stroke]{\pgfxy(34.00,10.00)}{0.50mm}
\pgfcircle[fill]{\pgfxy(38.00,10.00)}{0.50mm}
\pgfcircle[stroke]{\pgfxy(38.00,10.00)}{0.50mm}
\pgfmoveto{\pgfxy(32.00,20.00)}\pgflineto{\pgfxy(30.00,10.00)}\pgflineto{\pgfxy(38.00,10.00)}\pgfclosepath\pgfstroke
\pgfmoveto{\pgfxy(60.00,8.00)}\pgflineto{\pgfxy(60.00,33.00)}\pgfstroke
\pgfmoveto{\pgfxy(60.00,33.00)}\pgflineto{\pgfxy(59.30,30.20)}\pgflineto{\pgfxy(60.00,33.00)}\pgflineto{\pgfxy(60.70,30.20)}\pgflineto{\pgfxy(60.00,33.00)}\pgfclosepath\pgffill
\pgfmoveto{\pgfxy(60.00,33.00)}\pgflineto{\pgfxy(59.30,30.20)}\pgflineto{\pgfxy(60.00,33.00)}\pgflineto{\pgfxy(60.70,30.20)}\pgflineto{\pgfxy(60.00,33.00)}\pgfclosepath\pgfstroke
\pgfmoveto{\pgfxy(58.00,10.00)}\pgflineto{\pgfxy(114.00,10.00)}\pgfstroke
\pgfmoveto{\pgfxy(114.00,10.00)}\pgflineto{\pgfxy(111.20,10.70)}\pgflineto{\pgfxy(114.00,10.00)}\pgflineto{\pgfxy(111.20,9.30)}\pgflineto{\pgfxy(114.00,10.00)}\pgfclosepath\pgffill
\pgfmoveto{\pgfxy(114.00,10.00)}\pgflineto{\pgfxy(111.20,10.70)}\pgflineto{\pgfxy(114.00,10.00)}\pgflineto{\pgfxy(111.20,9.30)}\pgflineto{\pgfxy(114.00,10.00)}\pgfclosepath\pgfstroke
\pgfcircle[fill]{\pgfxy(60.00,10.00)}{0.50mm}
\pgfcircle[stroke]{\pgfxy(60.00,10.00)}{0.50mm}
\pgfcircle[fill]{\pgfxy(62.00,16.00)}{0.50mm}
\pgfcircle[stroke]{\pgfxy(62.00,16.00)}{0.50mm}
\pgfcircle[fill]{\pgfxy(64.00,22.00)}{0.50mm}
\pgfcircle[stroke]{\pgfxy(64.00,22.00)}{0.50mm}
\pgfcircle[fill]{\pgfxy(66.00,28.00)}{0.50mm}
\pgfcircle[stroke]{\pgfxy(66.00,28.00)}{0.50mm}
\pgfcircle[fill]{\pgfxy(68.00,26.00)}{0.50mm}
\pgfcircle[stroke]{\pgfxy(68.00,26.00)}{0.50mm}
\pgfcircle[fill]{\pgfxy(70.00,24.00)}{0.50mm}
\pgfcircle[stroke]{\pgfxy(70.00,24.00)}{0.50mm}
\pgfcircle[fill]{\pgfxy(72.00,22.00)}{0.50mm}
\pgfcircle[stroke]{\pgfxy(72.00,22.00)}{0.50mm}
\pgfcircle[fill]{\pgfxy(74.00,20.00)}{0.50mm}
\pgfcircle[stroke]{\pgfxy(74.00,20.00)}{0.50mm}
\pgfcircle[fill]{\pgfxy(76.00,26.00)}{0.50mm}
\pgfcircle[stroke]{\pgfxy(76.00,26.00)}{0.50mm}
\pgfcircle[fill]{\pgfxy(78.00,32.00)}{0.50mm}
\pgfcircle[stroke]{\pgfxy(78.00,32.00)}{0.50mm}
\pgfcircle[fill]{\pgfxy(80.00,30.00)}{0.50mm}
\pgfcircle[stroke]{\pgfxy(80.00,30.00)}{0.50mm}
\pgfcircle[fill]{\pgfxy(82.00,28.00)}{0.50mm}
\pgfcircle[stroke]{\pgfxy(82.00,28.00)}{0.50mm}
\pgfcircle[fill]{\pgfxy(84.00,26.00)}{0.50mm}
\pgfcircle[stroke]{\pgfxy(84.00,26.00)}{0.50mm}
\pgfcircle[fill]{\pgfxy(86.00,24.00)}{0.50mm}
\pgfcircle[stroke]{\pgfxy(86.00,24.00)}{0.50mm}
\pgfcircle[fill]{\pgfxy(88.00,22.00)}{0.50mm}
\pgfcircle[stroke]{\pgfxy(88.00,22.00)}{0.50mm}
\pgfcircle[fill]{\pgfxy(90.00,20.00)}{0.50mm}
\pgfcircle[stroke]{\pgfxy(90.00,20.00)}{0.50mm}
\pgfcircle[fill]{\pgfxy(92.00,18.00)}{0.50mm}
\pgfcircle[stroke]{\pgfxy(92.00,18.00)}{0.50mm}
\pgfcircle[fill]{\pgfxy(94.00,16.00)}{0.50mm}
\pgfcircle[stroke]{\pgfxy(94.00,16.00)}{0.50mm}
\pgfcircle[fill]{\pgfxy(96.00,14.00)}{0.50mm}
\pgfcircle[stroke]{\pgfxy(96.00,14.00)}{0.50mm}
\pgfcircle[fill]{\pgfxy(98.00,12.00)}{0.50mm}
\pgfcircle[stroke]{\pgfxy(98.00,12.00)}{0.50mm}
\pgfcircle[fill]{\pgfxy(100.00,18.00)}{0.50mm}
\pgfcircle[stroke]{\pgfxy(100.00,18.00)}{0.50mm}
\pgfcircle[fill]{\pgfxy(102.00,16.00)}{0.50mm}
\pgfcircle[stroke]{\pgfxy(102.00,16.00)}{0.50mm}
\pgfcircle[fill]{\pgfxy(104.00,14.00)}{0.50mm}
\pgfcircle[stroke]{\pgfxy(104.00,14.00)}{0.50mm}
\pgfcircle[fill]{\pgfxy(106.00,12.00)}{0.50mm}
\pgfcircle[stroke]{\pgfxy(106.00,12.00)}{0.50mm}
\pgfcircle[fill]{\pgfxy(108.00,10.00)}{0.50mm}
\pgfcircle[stroke]{\pgfxy(108.00,10.00)}{0.50mm}
\color[rgb]{1,1,1}\pgfcircle[fill]{\pgfxy(110.00,8.00)}{0.50mm}
\color[rgb]{0,0,0}\pgfcircle[stroke]{\pgfxy(110.00,8.00)}{0.50mm}
\pgfmoveto{\pgfxy(60.00,10.00)}\pgflineto{\pgfxy(62.00,16.00)}\pgflineto{\pgfxy(64.00,22.00)}\pgflineto{\pgfxy(66.00,28.00)}\pgflineto{\pgfxy(68.00,26.00)}\pgflineto{\pgfxy(70.00,24.00)}\pgflineto{\pgfxy(72.00,22.00)}\pgflineto{\pgfxy(74.00,20.00)}\pgflineto{\pgfxy(76.00,26.00)}\pgflineto{\pgfxy(78.00,32.00)}\pgflineto{\pgfxy(80.00,30.00)}\pgflineto{\pgfxy(82.00,28.00)}\pgflineto{\pgfxy(84.00,26.00)}\pgflineto{\pgfxy(88.00,22.00)}\pgflineto{\pgfxy(88.00,22.00)}\pgflineto{\pgfxy(90.00,20.00)}\pgflineto{\pgfxy(92.00,18.00)}\pgflineto{\pgfxy(94.00,16.00)}\pgflineto{\pgfxy(96.00,14.00)}\pgflineto{\pgfxy(98.00,12.00)}\pgflineto{\pgfxy(100.00,18.00)}\pgflineto{\pgfxy(102.00,16.00)}\pgflineto{\pgfxy(104.00,14.00)}\pgflineto{\pgfxy(106.00,12.00)}\pgflineto{\pgfxy(108.00,10.00)}\pgfstroke
\pgfmoveto{\pgfxy(110.00,8.00)}\pgflineto{\pgfxy(110.00,8.00)}\pgfstroke
\pgfsetdash{{0.30mm}{0.50mm}}{0mm}\pgfmoveto{\pgfxy(110.00,8.00)}\pgflineto{\pgfxy(108.00,10.00)}\pgfstroke
\end{pgfpicture}%
\caption{The bijection $\psi$}
\label{fig:psi}
\end{figure}
\nt{detail proof?}

\subsection*{Step 1}
Given $(T,v) \in \mathcal{V}$,
let $D_v$ be the subtree consisting of $v$ and all its descendants in $T$, say the \emph{descendant subtree of $v$}.
Letting $\ell'(\ge\ell)$ be the level of $v$, consider the path from $v$ to the root $r$ of $T$
$$v (=v_0) \to v_1 \to \dots \to v_{\ell} \to \dots \to v_{\ell'-1} \to r(=v_{\ell'}).$$
Record the number $p_i$ of elder siblings of $v_{i}$ in the tuplet of $v_i$ for all $0 \le i \le \ell-1$.
For all $0\le i \le \ell-1$, if $w_{i}$ is the youngest sibling of $v_i$ in the tuplet of $v_i$, we exchange two subtrees $D_{v_i}$ and $D_{w_i}$ and we obtain the tree $T'$.

\subsection*{Step 2}
For all $1 \le i \le \ell-1$ and $i=\ell'$, let $R_i$ be the subtree consisting $v_i$ and all its descendants on the right of the tuplet of $v_{i-1}$ in $T'$.
We obtain the tree $L$ by cutting the $\ell+1$ subtrees $D_{v}, R_1, \dots, R_{\ell-1}, R_{\ell'}$ from the tree $T'$, see Figure~\ref{fig:decomposition}.

%

%
%

\subsection*{A construction of the bijection $\Phi$}
We will construct the bijection $\Phi$ between $\mathcal{V}$ to $\mathcal{P} \times \mathcal{L}$.
Given $(T,v) \in \mathcal{V}$, let $k' (\ge k)$ be the outdegree of $v$ in $T$ and let $\ell' (\ge \ell)$ be the level of $v$ in $T$.
We separate two cases:

\subsubsection*{Case I}

If $v$ is not the root of $T$, i.e., $\ell' >0$.
We obtain the sequence $p = (p_0, \dots, p_{\ell-1}) \in \mathcal{P}$ in Step 1 and $(\ell+2)$ trees $D_v, R_1, R_2, \dots, R_{\ell-1}, R_{\ell'}, L$ after Step 2 as Figure~\ref{fig:decomposition}.
\begin{figure}[t]
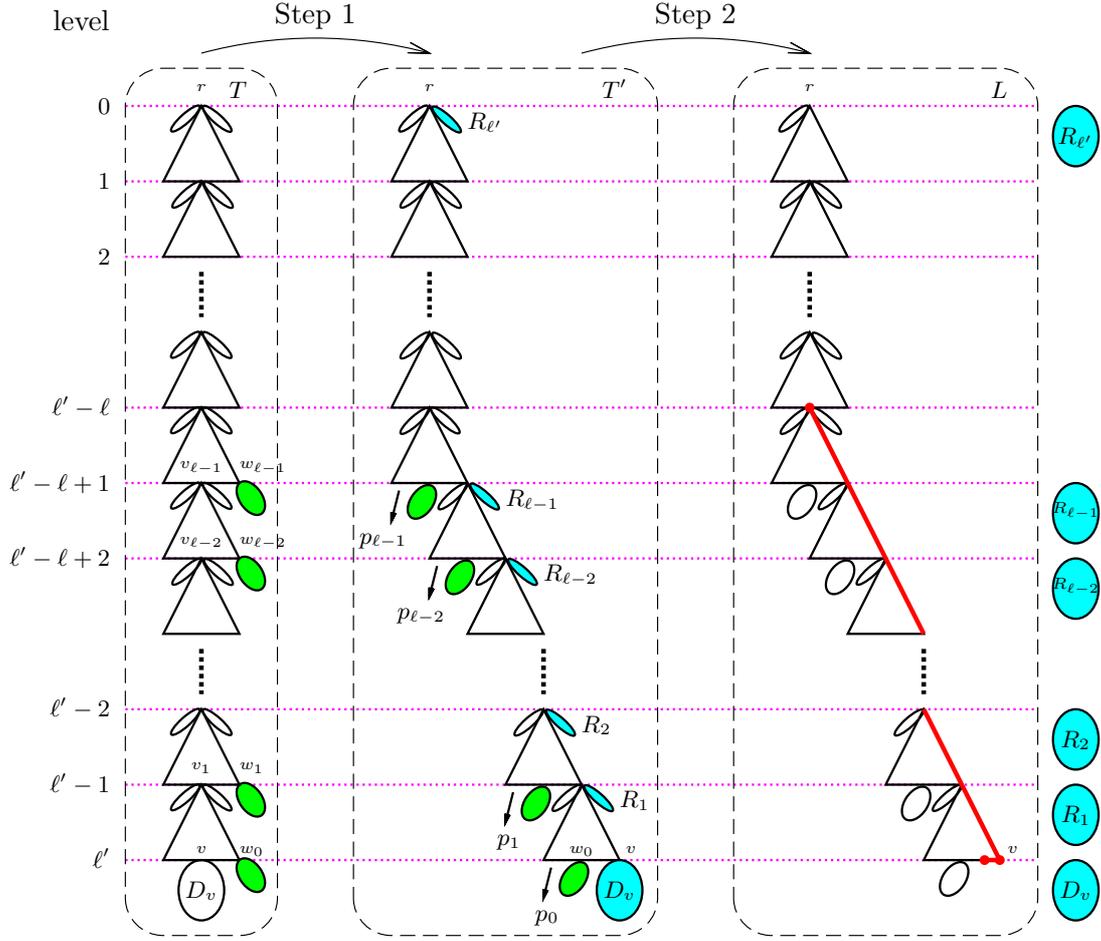

\centering
%
\caption{Tree decomposition}
\label{fig:decomposition}
\end{figure}

Let $\rho$ be the mapping on the set of lattice paths defined by
$$\rho(s_1 s_2 \cdots s_{n}) = s_2 \cdots s_n s_1,$$
where each $s_i$ is a step. Note that $\rho^m$ means to apply $\rho$ recursively $m$ times.


Clearly, the outdegree of the root of $D_v$ is $k'$.
In the tree $L$, there are no younger siblings of $v$ in the tuplet of $v$ and the outdegree of vertex $v$ is $0$.
Thus the lattice path $\rho^{a+\ell}(\overline{\varphi}(L))$ ends with one down-step and $\ell$ consecutive up-steps, 
where $a$ is the number of vertices of $L$ which precede $v$ in preorder.

We define a lattice path $P$ from $(0,0)$ to $((d+1)n+(\ell+1), -(\ell+1))$ by
$$P = \psi(D_v) \searrow \varphi(R_1) \searrow \varphi(R_2) \searrow \dots \searrow \varphi(R_{\ell-1}) \searrow \varphi(R_{\ell'}) \searrow \rho^{a+\ell}(\overline{\varphi}(L)),$$
where $\searrow$ means a down-step.

\subsubsection*{Case II}
If $v$ is the root of $T$, i.e., $\ell' = 0$.
We define a sequence $p = () \in \mathcal{P}$ and a lattice path $$P = \psi(T) \searrow.$$

In all cases, the lattice path $P$ always starts with at least $k$ (precisely $k'$) consecutive up-steps and ends with one down-step and $\ell$ consecutive up-steps as red segments in Figure~\ref{fig:mainlemma}.

\begin{figure}[t]
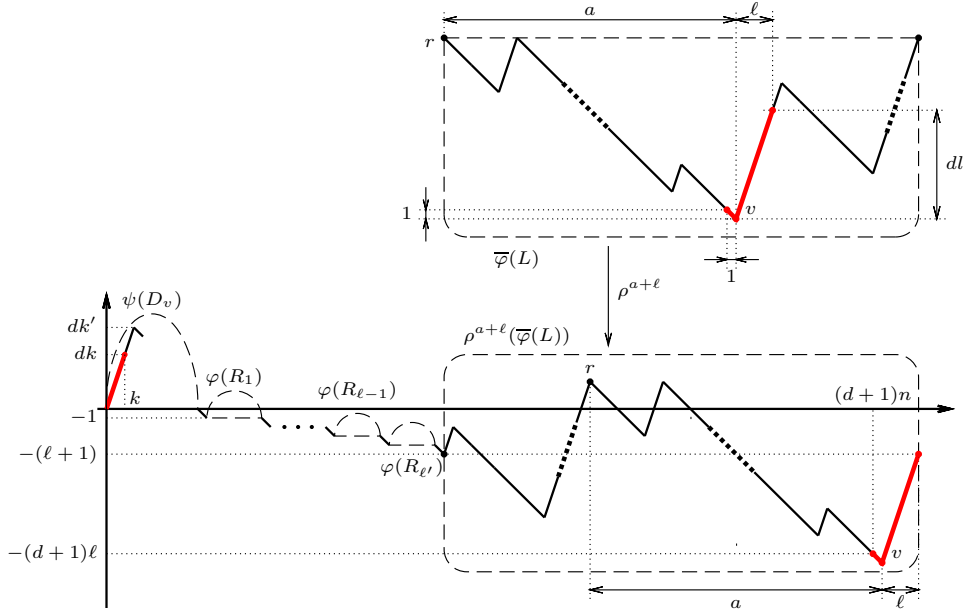

\centering
%
\caption{Outline of a lattice path $P$ induced from tree decomposition}
\label{fig:mainlemma}
\end{figure}

By removing the first $k$ steps and the last $(\ell+1)$ steps from $P$, we obtain the lattice path $\hat{P}$ of length $((d+1)n-k)$ from $(k, dk)$ to $((d+1)n, -(d+1)\ell)$, consisting of $(n-k-\ell)$ up-steps along the vector $(1,d)$ and $(dn+\ell)$ down-steps along the vector $(1,-1)$, so $\hat{P}$ belongs to $\mathcal{L}$.


Hence the map $\Phi: \mathcal{V} \to \mathcal{P} \times \mathcal{L}$ is defined by
$$\Phi(T,v) = (p, \hat{P}).$$

\subsection*{A description of the bijection $\Phi^{-1}$}
In the Case I of the construction of the bijection $\Phi$, given a lattice path $P$ from $(0,0)$ to $((d+1)n+(\ell+1), -(\ell+1))$, we decompose $P$ into $(\ell+2)$ paths $P_D, P_1, \dots, P_{\ell-1}, P_{\ell'}, P_L$ by removing the leftmost down-steps from height $-i$ to height $-(i+1)$ for $0\le i \le \ell$. Some of those paths may be empty.

Clearly all the paths $P_D, P_1, \dots, P_{\ell-1}, P_{\ell'}$ are $d$-Fuss-Catalan path.
By moving all the steps after the leftmost highest vertex in the lattice path $P_L$ to the beginning, we obtain a reverse $d$-Fuss-Catalan path $\overline{P}_L$ from $P_L$.
Since $\varphi$, $\overline{\varphi}$, and $\psi$ are bijections,
we can restore trees $D_v, R_1, \dots, R_{\ell-1}, R_{\ell'}, L$ from $P_D, P_1, \dots, P_{\ell-1}, P_{\ell'}, \overline{P}_L$.

Therefore, $\Phi$ is a bijection between $ \mathcal{V}$ and $ \mathcal{P} \times \mathcal{L}$
since all the remaining processes are also reversible.

\section{Proof of Theorem~\ref{thm:main2}}
\label{sec:calculation}

For any three nonnegative integers $i$, $j$, $k$ and one positive integer $\ell$, denote by $\Vnd(i,j,k;\ell)$ the set of pairs $(T,v)$
whose tree $T$ in $\T_n^{(d)}$ and vertex $v$ in $T$ such that
\begin{itemize}
\item $v$ has at least $i$ elder siblings in $T$,
\item $v$ has at least $j$ younger siblings in $T$,
\item $v$ has at least $k$ children in $T$,
\item $v$ is at level $\ge \ell$ in $T$.
\end{itemize}
We show the following lemma, which is a special case of Theorem~\ref{thm:main2}, that is, $i$ and $j$ are multiples of $d$.
\begin{lem}
\label{eq:lemma}
Given $n \ge 1$, for any three nonnegative integers $i$, $j$, $k$, all of which are multiples of $d$, and one positive integer $\ell$,
the cardinality of $\Vnd(i,j,k;\ell)$ is
\begin{align*}
d^\ell \binom{(d+1)n-\alpha}{dn+\ell},
\end{align*}
where $\alpha$ is the nonnegative integer satisfying $i+j+k = \alpha d$.
\end{lem}

\begin{proof}
That a vertex $v$ has at least $i$ elder (or younger resp.) siblings means that there exists at least $i/d$ (or $j/d$ resp.) $d$-tuplets directly connected from the parent of $v$ on its left (or right resp.).


A pair $(T,v)$ in $\Vnd(i,j,k,\ell)$
corresponds to a pair $(T',v)$ in $\Vnd(0,0,i+j+k,\ell)$
under a \emph{cut-and-paste} bijection $\gamma_{i,j}:(T, v) \mapsto (T',v)$
which cuts the leftmost $i/d$ tuplets connected from the parent $p$ of $v$
and pastes them at $v$ on the left and does again the rightmost $j/d$ tuplets connected from the parent $p$ of $v$ on the right, as Figure~\ref{fig:cut-and-paste}.
\begin{figure}
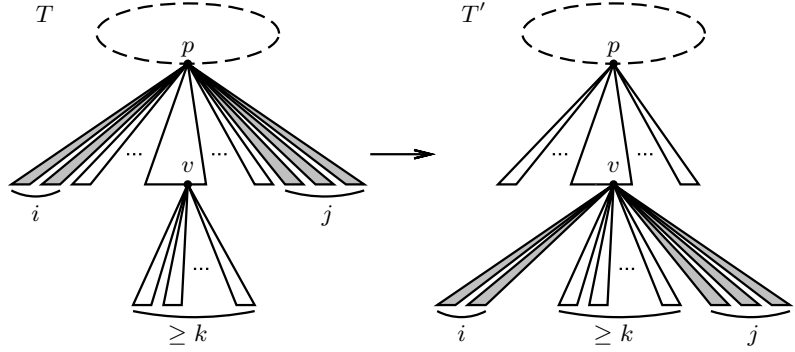

\centering
\begin{pgfpicture}{-9.20mm}{-7.49mm}{97.20mm}{47.60mm}
\pgfsetxvec{\pgfpoint{0.80mm}{0mm}}
\pgfsetyvec{\pgfpoint{0mm}{0.80mm}}
\color[rgb]{0,0,0}\pgfsetlinewidth{0.30mm}\pgfsetdash{}{0mm}
\color[rgb]{1,1,1}\pgfellipse[fill]{\pgfxy(31.00,57.00)}{\pgfxy(0.00,0.00)}{\pgfxy(0.00,0.00)}
\color[rgb]{0,0,0}\pgfellipse[stroke]{\pgfxy(31.00,57.00)}{\pgfxy(0.00,0.00)}{\pgfxy(0.00,0.00)}
\pgfcircle[fill]{\pgfxy(20.00,40.00)}{0.40mm}
\pgfcircle[stroke]{\pgfxy(20.00,40.00)}{0.40mm}
\pgfmoveto{\pgfxy(20.00,40.00)}\pgflineto{\pgfxy(13.00,20.00)}\pgflineto{\pgfxy(23.00,20.00)}\pgfclosepath\pgfstroke
\color[rgb]{0.75294,0.75294,0.75294}\pgfmoveto{\pgfxy(20.00,40.00)}\pgflineto{\pgfxy(-9.00,20.00)}\pgflineto{\pgfxy(-6.00,20.00)}\pgfclosepath\pgffill
\color[rgb]{0,0,0}\pgfmoveto{\pgfxy(20.00,40.00)}\pgflineto{\pgfxy(-9.00,20.00)}\pgflineto{\pgfxy(-6.00,20.00)}\pgfclosepath\pgfstroke
\color[rgb]{0.75294,0.75294,0.75294}\pgfmoveto{\pgfxy(20.00,40.00)}\pgflineto{\pgfxy(-4.00,20.00)}\pgflineto{\pgfxy(-1.00,20.00)}\pgfclosepath\pgffill
\color[rgb]{0,0,0}\pgfmoveto{\pgfxy(20.00,40.00)}\pgflineto{\pgfxy(-4.00,20.00)}\pgflineto{\pgfxy(-1.00,20.00)}\pgfclosepath\pgfstroke
\pgfmoveto{\pgfxy(20.00,40.00)}\pgflineto{\pgfxy(1.00,20.00)}\pgflineto{\pgfxy(4.00,20.00)}\pgfclosepath\pgfstroke
\pgfmoveto{\pgfxy(20.00,40.00)}\pgflineto{\pgfxy(31.00,20.00)}\pgflineto{\pgfxy(34.00,20.00)}\pgfclosepath\pgfstroke
\color[rgb]{0.75294,0.75294,0.75294}\pgfmoveto{\pgfxy(20.00,40.00)}\pgflineto{\pgfxy(36.00,20.00)}\pgflineto{\pgfxy(39.00,20.00)}\pgfclosepath\pgffill
\color[rgb]{0,0,0}\pgfmoveto{\pgfxy(20.00,40.00)}\pgflineto{\pgfxy(36.00,20.00)}\pgflineto{\pgfxy(39.00,20.00)}\pgfclosepath\pgfstroke
\color[rgb]{0.75294,0.75294,0.75294}\pgfmoveto{\pgfxy(20.00,40.00)}\pgflineto{\pgfxy(41.00,20.00)}\pgflineto{\pgfxy(44.00,20.00)}\pgfclosepath\pgffill
\color[rgb]{0,0,0}\pgfmoveto{\pgfxy(20.00,40.00)}\pgflineto{\pgfxy(41.00,20.00)}\pgflineto{\pgfxy(44.00,20.00)}\pgfclosepath\pgfstroke
\color[rgb]{0.75294,0.75294,0.75294}\pgfmoveto{\pgfxy(20.00,40.00)}\pgflineto{\pgfxy(46.00,20.00)}\pgflineto{\pgfxy(49.00,20.00)}\pgfclosepath\pgffill
\color[rgb]{0,0,0}\pgfmoveto{\pgfxy(20.00,40.00)}\pgflineto{\pgfxy(46.00,20.00)}\pgflineto{\pgfxy(49.00,20.00)}\pgfclosepath\pgfstroke
\pgfsetdash{{0.30mm}{0.50mm}}{0mm}\pgfmoveto{\pgfxy(10.00,25.00)}\pgflineto{\pgfxy(13.00,25.00)}\pgfstroke
\pgfmoveto{\pgfxy(24.00,25.00)}\pgflineto{\pgfxy(27.00,25.00)}\pgfstroke
\pgfcircle[fill]{\pgfxy(90.00,40.00)}{0.40mm}
\pgfsetdash{}{0mm}\pgfcircle[stroke]{\pgfxy(90.00,40.00)}{0.40mm}
\pgfmoveto{\pgfxy(90.00,40.00)}\pgflineto{\pgfxy(83.00,20.00)}\pgflineto{\pgfxy(93.00,20.00)}\pgfclosepath\pgfstroke
\color[rgb]{0.75294,0.75294,0.75294}\pgfmoveto{\pgfxy(90.00,20.00)}\pgflineto{\pgfxy(66.00,0.00)}\pgflineto{\pgfxy(69.00,0.00)}\pgfclosepath\pgffill
\color[rgb]{0,0,0}\pgfmoveto{\pgfxy(90.00,20.00)}\pgflineto{\pgfxy(66.00,0.00)}\pgflineto{\pgfxy(69.00,0.00)}\pgfclosepath\pgfstroke
\color[rgb]{0.75294,0.75294,0.75294}\pgfmoveto{\pgfxy(90.00,20.00)}\pgflineto{\pgfxy(61.00,0.00)}\pgflineto{\pgfxy(64.00,0.00)}\pgfclosepath\pgffill
\color[rgb]{0,0,0}\pgfmoveto{\pgfxy(90.00,20.00)}\pgflineto{\pgfxy(61.00,0.00)}\pgflineto{\pgfxy(64.00,0.00)}\pgfclosepath\pgfstroke
\pgfmoveto{\pgfxy(90.00,40.00)}\pgflineto{\pgfxy(71.00,20.00)}\pgflineto{\pgfxy(74.00,20.00)}\pgfclosepath\pgfstroke
\pgfmoveto{\pgfxy(90.00,40.00)}\pgflineto{\pgfxy(101.00,20.00)}\pgflineto{\pgfxy(104.00,20.00)}\pgfclosepath\pgfstroke
\color[rgb]{0.75294,0.75294,0.75294}\pgfmoveto{\pgfxy(90.00,20.00)}\pgflineto{\pgfxy(106.00,0.00)}\pgflineto{\pgfxy(109.00,0.00)}\pgfclosepath\pgffill
\color[rgb]{0,0,0}\pgfmoveto{\pgfxy(90.00,20.00)}\pgflineto{\pgfxy(106.00,0.00)}\pgflineto{\pgfxy(109.00,0.00)}\pgfclosepath\pgfstroke
\color[rgb]{0.75294,0.75294,0.75294}\pgfmoveto{\pgfxy(90.00,20.00)}\pgflineto{\pgfxy(111.00,0.00)}\pgflineto{\pgfxy(114.00,0.00)}\pgfclosepath\pgffill
\color[rgb]{0,0,0}\pgfmoveto{\pgfxy(90.00,20.00)}\pgflineto{\pgfxy(111.00,0.00)}\pgflineto{\pgfxy(114.00,0.00)}\pgfclosepath\pgfstroke
\color[rgb]{0.75294,0.75294,0.75294}\pgfmoveto{\pgfxy(90.00,20.00)}\pgflineto{\pgfxy(116.00,0.00)}\pgflineto{\pgfxy(119.00,0.00)}\pgfclosepath\pgffill
\color[rgb]{0,0,0}\pgfmoveto{\pgfxy(90.00,20.00)}\pgflineto{\pgfxy(116.00,0.00)}\pgflineto{\pgfxy(119.00,0.00)}\pgfclosepath\pgfstroke
\color[rgb]{1,1,1}\pgfellipse[fill]{\pgfxy(87.00,20.00)}{\pgfxy(0.00,0.00)}{\pgfxy(0.00,0.00)}
\color[rgb]{0,0,0}\pgfellipse[stroke]{\pgfxy(87.00,20.00)}{\pgfxy(0.00,0.00)}{\pgfxy(0.00,0.00)}
\pgfsetdash{{0.30mm}{0.50mm}}{0mm}\pgfmoveto{\pgfxy(80.00,25.00)}\pgflineto{\pgfxy(83.00,25.00)}\pgfstroke
\pgfmoveto{\pgfxy(94.00,25.00)}\pgflineto{\pgfxy(97.00,25.00)}\pgfstroke
\pgfsetdash{{2.00mm}{1.00mm}}{0mm}\pgfellipse[stroke]{\pgfxy(90.00,45.00)}{\pgfxy(15.00,0.00)}{\pgfxy(0.00,5.00)}
\pgfellipse[stroke]{\pgfxy(20.00,45.00)}{\pgfxy(15.00,0.00)}{\pgfxy(0.00,5.00)}
\pgfputat{\pgfxy(20.00,22.00)}{\pgfbox[bottom,left]{\fontsize{9.10}{10.93}\selectfont \makebox[0pt]{$v$}}}
\pgfputat{\pgfxy(90.00,22.00)}{\pgfbox[bottom,left]{\fontsize{9.10}{10.93}\selectfont \makebox[0pt]{$v$}}}
\pgfputat{\pgfxy(-5.00,47.00)}{\pgfbox[bottom,left]{\fontsize{9.10}{10.93}\selectfont $T$}}
\pgfputat{\pgfxy(65.00,47.00)}{\pgfbox[bottom,left]{\fontsize{9.10}{10.93}\selectfont $T'$}}
\pgfsetdash{}{0mm}\pgfmoveto{\pgfxy(50.00,25.00)}\pgflineto{\pgfxy(60.00,25.00)}\pgfstroke
\pgfmoveto{\pgfxy(60.00,25.00)}\pgflineto{\pgfxy(57.20,25.70)}\pgflineto{\pgfxy(60.00,25.00)}\pgflineto{\pgfxy(57.20,24.30)}\pgflineto{\pgfxy(60.00,25.00)}\pgfclosepath\pgffill
\pgfmoveto{\pgfxy(60.00,25.00)}\pgflineto{\pgfxy(57.20,25.70)}\pgflineto{\pgfxy(60.00,25.00)}\pgflineto{\pgfxy(57.20,24.30)}\pgflineto{\pgfxy(60.00,25.00)}\pgfclosepath\pgfstroke
\pgfputat{\pgfxy(43.00,14.00)}{\pgfbox[bottom,left]{\fontsize{9.10}{10.93}\selectfont \makebox[0pt]{$j$}}}
\pgfmoveto{\pgfxy(36.00,19.00)}\pgfcurveto{\pgfxy(38.26,18.28)}{\pgfxy(40.63,17.95)}{\pgfxy(43.00,18.00)}\pgfcurveto{\pgfxy(45.04,18.05)}{\pgfxy(47.06,18.38)}{\pgfxy(49.00,19.00)}\pgfstroke
\pgfmoveto{\pgfxy(-9.00,19.00)}\pgfcurveto{\pgfxy(-7.77,18.34)}{\pgfxy(-6.40,18.00)}{\pgfxy(-5.00,18.00)}\pgfcurveto{\pgfxy(-3.60,18.00)}{\pgfxy(-2.23,18.34)}{\pgfxy(-1.00,19.00)}\pgfstroke
\pgfputat{\pgfxy(-5.00,14.00)}{\pgfbox[bottom,left]{\fontsize{9.10}{10.93}\selectfont \makebox[0pt]{$i$}}}
\pgfmoveto{\pgfxy(61.00,-1.00)}\pgfcurveto{\pgfxy(62.23,-1.66)}{\pgfxy(63.60,-2.00)}{\pgfxy(65.00,-2.00)}\pgfcurveto{\pgfxy(66.40,-2.00)}{\pgfxy(67.77,-1.66)}{\pgfxy(69.00,-1.00)}\pgfstroke
\pgfputat{\pgfxy(65.00,-6.00)}{\pgfbox[bottom,left]{\fontsize{9.10}{10.93}\selectfont \makebox[0pt]{$i$}}}
\pgfputat{\pgfxy(113.00,-6.00)}{\pgfbox[bottom,left]{\fontsize{9.10}{10.93}\selectfont \makebox[0pt]{$j$}}}
\pgfmoveto{\pgfxy(106.00,-1.00)}\pgfcurveto{\pgfxy(108.26,-1.72)}{\pgfxy(110.63,-2.05)}{\pgfxy(113.00,-2.00)}\pgfcurveto{\pgfxy(115.04,-1.95)}{\pgfxy(117.06,-1.62)}{\pgfxy(119.00,-1.00)}\pgfstroke
\pgfputat{\pgfxy(20.00,42.00)}{\pgfbox[bottom,left]{\fontsize{9.10}{10.93}\selectfont \makebox[0pt]{$p$}}}
\pgfputat{\pgfxy(90.00,42.00)}{\pgfbox[bottom,left]{\fontsize{9.10}{10.93}\selectfont \makebox[0pt]{$p$}}}
\pgfcircle[fill]{\pgfxy(20.00,20.00)}{0.40mm}
\pgfcircle[stroke]{\pgfxy(20.00,20.00)}{0.40mm}
\pgfmoveto{\pgfxy(20.00,20.00)}\pgflineto{\pgfxy(11.00,0.00)}\pgflineto{\pgfxy(14.00,0.00)}\pgfclosepath\pgfstroke
\pgfputat{\pgfxy(20.00,-6.00)}{\pgfbox[bottom,left]{\fontsize{9.10}{10.93}\selectfont \makebox[0pt]{$\ge k$}}}
\pgfmoveto{\pgfxy(11.00,-1.00)}\pgfcurveto{\pgfxy(13.96,-1.60)}{\pgfxy(16.98,-1.94)}{\pgfxy(20.00,-2.00)}\pgfcurveto{\pgfxy(23.69,-2.07)}{\pgfxy(27.38,-1.74)}{\pgfxy(31.00,-1.00)}\pgfstroke
\pgfmoveto{\pgfxy(20.00,20.00)}\pgflineto{\pgfxy(16.00,0.00)}\pgflineto{\pgfxy(19.00,0.00)}\pgflineto{\pgfxy(20.00,20.00)}\pgfclosepath\pgfstroke
\pgfmoveto{\pgfxy(20.00,20.00)}\pgflineto{\pgfxy(31.00,0.00)}\pgflineto{\pgfxy(28.00,0.00)}\pgflineto{\pgfxy(20.00,20.00)}\pgfclosepath\pgfstroke
\pgfsetdash{{0.30mm}{0.50mm}}{0mm}\pgfmoveto{\pgfxy(21.00,6.00)}\pgflineto{\pgfxy(24.00,6.00)}\pgfstroke
\pgfcircle[fill]{\pgfxy(90.00,20.00)}{0.40mm}
\pgfsetdash{}{0mm}\pgfcircle[stroke]{\pgfxy(90.00,20.00)}{0.40mm}
\pgfmoveto{\pgfxy(90.00,20.00)}\pgflineto{\pgfxy(81.00,0.00)}\pgflineto{\pgfxy(84.00,0.00)}\pgfclosepath\pgfstroke
\pgfputat{\pgfxy(90.00,-6.00)}{\pgfbox[bottom,left]{\fontsize{9.10}{10.93}\selectfont \makebox[0pt]{$\ge k$}}}
\pgfmoveto{\pgfxy(81.00,-1.00)}\pgfcurveto{\pgfxy(83.96,-1.60)}{\pgfxy(86.98,-1.94)}{\pgfxy(90.00,-2.00)}\pgfcurveto{\pgfxy(93.69,-2.07)}{\pgfxy(97.38,-1.74)}{\pgfxy(101.00,-1.00)}\pgfstroke
\pgfmoveto{\pgfxy(90.00,20.00)}\pgflineto{\pgfxy(86.00,0.00)}\pgflineto{\pgfxy(89.00,0.00)}\pgflineto{\pgfxy(90.00,20.00)}\pgfclosepath\pgfstroke
\pgfmoveto{\pgfxy(90.00,20.00)}\pgflineto{\pgfxy(101.00,0.00)}\pgflineto{\pgfxy(98.00,0.00)}\pgflineto{\pgfxy(90.00,20.00)}\pgfclosepath\pgfstroke
\pgfsetdash{{0.30mm}{0.50mm}}{0mm}\pgfmoveto{\pgfxy(91.00,6.00)}\pgflineto{\pgfxy(94.00,6.00)}\pgfstroke
\end{pgfpicture}%
\caption{Cut-and-paste bijection $\gamma_{i,j}$}
\label{fig:cut-and-paste}
\end{figure}

Since that $v$ has at least $i+j+k$ children means that the outdegree of $v$ greater than or equal to $\alpha = \dfrac{i+j+k}{d}$, this case corresponds to $k \leftarrow \alpha$ of Theorem~\ref{thm:main1}.
\end{proof}

In Theorem~\ref{thm:main2}, what to find is the cardinality of $\Vnd(i,j,k;\ell)$ for any two nonnegative integers $i$, $j$, one nonnegative integer $k$ which is a multiple of $d$, and one positive integer $\ell$.

Given $(T, v) \in \Vnd(i,j,k;\ell)$, let $w$ be the $j$th younger sibling of $v$.
By exchanging two subtrees $D_v$ and $D_w$,
we obtain $(T', v)$ in $\Vnd(i+j,0,k;\ell)$ from $(T, v)$ in $\Vnd(i,j,k;\ell)$.
Let $\alpha$ and $\beta$ be the quotient and the remainder when $i+j+k$ is divided by $d$, that is, $$i+j+k = \alpha d +\beta.$$
By applying the cut-and-paste bijection $\gamma_{i+j-\beta, 0}$,
we obtain $(T'',v)$ in  $\Vnd(\beta,0,\alpha d;\ell)$ from $(T', v)$ in $\Vnd(i+j,0,k;\ell)$.
One can show that the values
$$\#\Vnd(i,0,\alpha d;\ell)-\#\Vnd(i+1,0,\alpha d;\ell)$$
are the same for all $0\le i \le d-1$
under exchanging two descendant subtrees of two sibling in the same tuplet.
By telescoping, we get the formula
\begin{align*}
&\#\Vnd(0,0,\alpha d;\ell)-\#\Vnd(\beta,0,\alpha d;\ell)\\
&=
\frac{\beta}{d}
\left[
\#\Vnd(0,0,\alpha d;\ell)
- \#\Vnd(d,0,\alpha d;\ell)
\right].
\end{align*}
By Lemma~\ref{eq:lemma}, we have
\begin{align*}
\#\Vnd(0,0,\alpha d;\ell) &= d^\ell \binom{(d+1)n-\alpha}{dn+\ell}, \\
\#\Vnd(d,0,\alpha d;\ell) &=  d^\ell \binom{(d+1)n-\alpha-1}{dn+\ell}.
\end{align*}
Thus we get the cardinality of $\Vnd(\beta,0,\alpha d;\ell)$ and the desired formula~\eqref{eq:refinement}.







\section{Further results}
\label{sec:coro}
From Theorem~\ref{thm:main1}, we can obtain the following result.
\begin{cor}
Given $n \ge 1$, for any two nonnegative integers $k$ and $\ell$, the number of all vertices of outdegree $k$ at level $\ell$ among $d$-trees in $\mathcal{T}_n^{(d)}$ is
\begin{align}
\label{eq:coro}
d^\ell \frac{dk+(d+1)\ell}{(d+1)n-k} \binom{(d+1)n-k}{dn+\ell}.
\end{align}
\end{cor}

\begin{proof}
By the sieve method with \eqref{eq:KSS16}, we obtain the formula \eqref{eq:coro} from
\begin{multline*}
d^\ell \binom{(d+1)n-k}{dn+\ell}
- d^\ell \binom{(d+1)n-k-1}{dn+\ell} \\
- d^{\ell+1} \binom{(d+1)n-k}{dn+\ell+1}
+ d^{\ell+1} \binom{(d+1)n-k-1}{dn+\ell+1}.
\end{multline*}
\end{proof}

The next result follows from Theorem~\ref{thm:main2} for $d=1$.
\begin{cor}
Given $n \ge 1$, for any three nonnegative integers $i$, $j$, $k$, and one positive integer $\ell$,
the number of all vertices among trees in $\T_n$ such that
\begin{itemize}
\item having at least $i$ elder siblings,
\item having at least $j$ younger siblings,
\item having at least $k$ children,
\item at level $\ge \ell$
\end{itemize}
is
\begin{align*}
\binom{2n-i-j-k}{n+\ell}.
\end{align*}
\end{cor}

\section*{Acknowledgements}
For the third author, this work was supported by the National Research Foundation of Korea(NRF) grant funded by the Korea government(MSIP) (No. 2017R1C1B2008269).


\bibliographystyle{alpha}

\end{document}